\documentclass[11pt]{amsart}
\usepackage{a4}
\usepackage{amssymb}
\usepackage{array}
\usepackage{booktabs}
\usepackage{multirow}
\usepackage{hhline}
\usepackage[all]{xy}

\newtheorem{thm}{Theorem}

\newtheorem{cor}[thm]{Corollary}
\newtheorem{lem}[thm]{Lemma}
\newtheorem{prop}[thm]{Proposition}
\newtheorem{defn}[thm]{Definition}

\newtheorem{rem}[thm]{Remark}

\numberwithin{thm}{section}
\numberwithin{equation}{section}
\newcommand{\Real}{\mathbb R}

\newcommand{\norm}[1]{\left\Vert#1\right\Vert}
\newcommand{\abs}[1]{\left\vert#1\right\vert}

\newcommand{\eps}{\varepsilon}

\newcommand{\la}{\langle}
\newcommand{\ra}{\rangle}

\newcommand{\A}{\mathcal{A}}

\newcommand{\C}{\mathbb{C}}

\newcommand{\Fb}{\mathbb{F}}

\newcommand{\N}{\mathbb{N}}

\newcommand{\sg} {\sigma}

\newcommand{\Z}{\mathbb{Z}}

\newcommand{\Om}{\Omega}
\newcommand{\om}{\omega}

\begin{document}
\title[Some weighted group algebras are operator algebras]
{Some weighted group algebras are operator algebras}
\author{Hun Hee Lee}
\address{Department of Mathematical Sciences,
Seoul National University, San56-1 Shinrim-dong Kwanak-gu
Seoul 151-747, Republic of Korea}
\email{hhlee@chungbuk.ac.kr}

\author{Ebrahim Samei}
\address{Department of Mathematics and Statistics, University of Saskatchewan, Saskatoon, Saskatchewan, S7N 5E6, Canada}
\email{samei@math.usask.ca}

\author{Nico Spronk}
\address{Department of Pure Mathematics, University of Waterloo, Waterloo,
Ontario, N2L 3G1, Canada}
\email{nspronk@math.uwaterloo.ca}

\thanks{Hun Hee Lee was supported by Basic Science Research Program through the National
Research Foundation of Korea(NRF) funded by the Ministry of Education, Science and Technology (2010-
0015222). Ebrahim Samei was supported by NSERC under grant no. 366066-09. Nico Spronk
was supported by NSERC under grant no. 312515-05.}



\begin{abstract}
Let $G$ be a finitely generated group with polynomial growth, and let $\om$ be a weight, i.e. a sub-multiplicative function
on $G$ with positive values. We study when the weighted group algebra $\ell^1(G,\om)$ is isomorphic to an operator algebra.
We show that $\ell^1(G,\om)$ is isomorphic to an operator algebra if $\om$ is a polynomial weight with large enough degree or an exponential weight of order $0<\alpha<1$.
We will demonstrate the order of growth of $G$ plays an important role in this question.
Moreover, the algebraic centre of $\ell^1(G,\om)$ is isomorphic to a $Q$-algebra and hence
satisfies a multi-variable von Neumann inequality. We also present a more detailed study of
our results when $G$ is the $d$-dimensional integers $\Z^d$ and 3-dimensional discrete Heisenberg group $\mathbb{H}_3(\Z)$.
The case of the free group with two generators will be considered as a counter example of groups with exponential growth.
\end{abstract}

\maketitle

\section{Introduction}

The motivation for this paper was originated from a result of Varopoulos which states that
certain weighted group algebras on integers are isomorphic to $Q$-algebras \cite{V}.
We recall that a commutative Banach algebra is called a $Q$-algebra if it is a quotient of a uniform algebra.
There are interesting (and non-trivial) classes of Banach algebras which are isomorphic to $Q$-algebras.
For instance, it is shown in \cite{V} and \cite{Dav} that the spaces $\ell^p$ $(1\leq p \leq \infty)$ with pointwise product are isomorphic to  $Q$-algebras.
The case of the Schatten Spaces $S_p$, $1\leq p \leq \infty$, endowed with the Schur product, has been considered by many researchers (\cite{Le} and \cite{P-G}), and it has been very recently answered in full generality (\cite{BBLV}).

Let $G$ be a discrete group, and let $\om : G \to (0,\infty)$ be a weight on $G$, i.e.
	$$\om(xy)\leq \om(x)\om(y),\;\; x,y\in G.$$
The {\it weighted group algebra} $\ell^1(G,\om)$ is the convolution algebra of functions
$f$ on $G$ such that $\|f\|_{\ell^1(G,\om)}=\sum_{x\in G} |f(x)|\om(x) < \infty$.  Varopoulos showed that
in the case where $G=\Z$ and $\om_\alpha(n)=(1+|n|)^\alpha$ ($\alpha \ge 0$),
$\ell^1(\Z, \om_\alpha)$ is isomorphic to a $Q$-algebra if and only if $\alpha > 1/2$.

We would like to extend Varopoulos's result to other classes of weighted group algebras, possibly on non-abelian groups.
However, group algebras are non-commutative in general, so that we can not hope for them to be isomorphic to $Q$-algebras.
Instead, we would like to investigate whether a weighted group algebra is isomorphic to an operator algebra.
Recall that an {\it operator algebra} is a closed subalgebra of $B(H)$, the algebra of all bounded operators on a Hilbert space $H$.
Note that any $Q$-algebra is an operator algebra (\cite[Theorem 1.1]{Dav}).
In the proof, Varopoulos actually proved that $\ell^1(\Z, \om_\alpha)$ satisfies one of the sufficient conditions to be isomorphic to a $Q$-algebra, namely it is an {\it injective algebra}. Recall that a Banach algebra $\A$ is called an {\it injective algebra} if the algebra multiplication map $m$ extends to a bounded map on the injective tensor product:
	$$m: \A\otimes_\eps \A \to \A.$$
In this paper we also focus on the case where $\ell^1(G,\om)$ becomes an injective algebra.
Using a Littlewood multiplier argument we will show that $\ell^1(G,\om)$ is an injective algebra, and consequently is isomorphic to an operator algebra, if $G$ is a finitely generated group with polynomial growth and $\om$ is a polynomial weight with a large enough degree or a certain exponential weight. Such weights will be defined later in this paper.

This paper is organized as follows. In Section \ref{S:algebras} we recall several basic facts about injective algebras and $Q$-algebras. In Section \ref{S:weighted} we give an equivalence condition for $\ell^1(G,\om)$ to be isomorphic to an operator algebra.
In Section \ref{S:Littlewood mult} we recall the definitions of Littlewood multipliers and its consequences. In Section \ref{S:Groups poly. growth}, we give the necessary background on finitely generated groups with polynomial growth and how one can use the length function to define various weights such as polynomial and exponential weights on these groups.
In Sections \ref{S:Poly. weight-poly growth} and \ref{S:Exp. weight-poly growth}, we will show our main results, namely the case where $\ell^1(G,\om)$ is isormophic to an operator algebra.
Moreover, we will check that the algebraic centre of $\ell^1(G,\om)$ is a $Q$-algebra in this case, and hence,
it satisfies the $(\delta,L)$-multi-variable von Neumann inequality (see Section \ref{S:Poly. weight-poly growth}). We also find estimates for the upper bound of the norm of the multiplication map of the algebra
for various weights and use them to determine concrete values of $\delta$ and $L$.

Finally, in Section \ref{S:Examples}, we apply our techniques to study the cases when $G$ is the $d$-dimensional integers $\Z^d$ or 3-dimensional discrete Heisenberg group $\mathbb{H}_3(\Z)$. The case of the free group with two generators will be examined to give a reasonable explanation why we mainly focus on groups with polynomial growth.

\section{Preliminaries}\label{S:Preliminaries}

In this paper, all our groups are discrete.

\subsection{$p$-summing algebras, injective algebras and $Q$-algebras}\label{S:algebras}

We first recall some definitions. Let $X$ and $Y$ be Banach spaces.
For $1\le p <\infty$, a sequence $(x_n)_{n\ge 1} \subset X$ is called $p$-summable (resp. weakly $p$-summable) if
	$$\norm{(x_n)}_p = \left(\sum_{n\ge 1}\norm{x_n}^p\right)^{\frac{1}{p}} <\infty.\;\;
	(\text{resp.}\;\norm{(x_n)}^w_p = \sup_{\varphi \in B_{X^*}}\left(\sum_{n\ge 1}\abs{\varphi(x_n)}^p\right)^{\frac{1}{p}} <\infty.)$$
The Chevet-Saphar tensor norms on the algebraic tensor product $X\otimes Y$ are defined by
	$$g_p(u) = \inf\{\norm{(x_j)}_p\norm{(y_j)}^w_{p'} : u = \sum^n_{j=1}x_j\otimes y_j, \, x_j\in X, \, y_j\in Y\},$$
where $p'$ is the conjugate index of $p$.
We denote the completion of $(X\otimes Y, g_p)$  by $X\otimes_{g_p}Y$.

We say that a linear map $T : X\to Y$ is $p$-summing if there is a constant $C>0$ such that
	$$\norm{(Tx_n)}_p \le C\norm{(x_n)}^w_p$$
for any sequence $(x_n)_{n\ge 1} \subset X$.
We denote the infimum of such $C$ by $\pi_p(T)$, and
$\Pi_p(X,Y)$ refers to the Banach space of all $p$-summing maps with the norm $\pi_p(\cdot)$.
It is well-known that we have the following isometry
	$$(X\otimes_{g_p}Y)^* \cong \Pi_{p'}(Y,X^*),\; A\otimes B \mapsto T$$
where $A\in X^*$, $B\in Y^*$ and
	$$Ty = \la y, B\ra A, \ \ (x\in X, y\in Y).$$
See \cite[Chapter 6]{Ryan} for the details of $p$-summing maps and Chevet-Saphar tensor norms.

A standard Banach space theory (\cite[Proposition 3.22]{Ryan} and \cite[Corollary 9.5]{Tom})
tells us that we have the following isometry
\begin{equation}\label{eq-1-summing-injective tensor}
	(\ell^1(G) \otimes_\eps \ell^1(G))^* \cong \Pi_1(\ell^1(G), \ell^\infty(G)), \; A\otimes B \mapsto S.
\end{equation}
One more standard fact we will use later is that the composition of two 2-summing maps is a 1-summing map (actually, a nuclear map). More precisely, let $T: X\to Y$ and $S:Y\to Z$ be 2-summing maps between Banach spaces, then $S\circ T$ is 1-summing with
	\begin{equation}\label{eq-1-summing}
	\pi_1(S\circ T) \le \pi_2(S)\pi_2(T).
	\end{equation}

We say that a Banach algebra $\A$ is a {\it $p$-summing algebra} if the algebra multiplication map $m$ extends to a bounded map
	$$m : \A\otimes_{g_p}\A \to \A.$$
	\begin{thm}({\bf Tonge}, \cite[Theorem 18.19]{DJT})
	Every 2-summing algebra is isomorphic to an operator algebra.
	\end{thm}

	\begin{cor}\label{C:injective-op alg}
	Every injective algebra is isomorphic to an operator algebra.
	\end{cor}

\begin{proof}
Recall that the injective tensor product is the minimal among Banach space tensor products,
so that the formal identity $\A\otimes_{g_2}\A \to \A\otimes_\eps\A$ is a contraction for a Banach algebra $\A$.
Thus we can conclude that every injective algebra is a 2-summing algebra, which gives the conclusion we wanted.
\end{proof}
	
	\begin{defn}
	Let $m$ be the algebra multiplication of a Banach algebra $\A$. In the case that $\A$ is an injective algebra, we will denote
		$$\norm{m}_\eps := \norm{m : \A\otimes_\eps \A \to \A}.$$
	\end{defn}

We say a Banach algebra $\A$ is a {\it $Q$-algebra} if it is a quotient of a uniform algebra, which is automatically a commutative algebra. $Q$-algebras are characterized by a von Neumann type inequality \cite[Section 5.4.3(2)]{BLe}.

	\begin{thm}
	Let $\A$ be a commutative Banach algebra.
	Then $\A$ is isometrically isomorphic to a $Q$-algebra if and only if we have
		$$\|p(a_1,\ldots,a_n)\| \le \|p\|_\infty$$
	for any $n\in N$, $\{a_1,\ldots,a_n \}\subset A$ with norm $\le 1$ and every polynomial $p$ in $n$ variables without constant terms, where
		$$\|p\|_\infty=\sup\{|p(z_1,\ldots,z_n) \mid |z_i|\leq 1, i=1,\ldots,n \}.$$
	\end{thm}

Motivated by the above we give the following definition.

	\begin{defn}\label{D: multi-variable-von Neumann inequality}
	Let $A$ be a commutative Banach algebra. Then $A$ is said to satisfy
	{\it multi-variable ($\delta, L$)-von Neumann inequality} provided that for
	every $n\in N$, every set of $n$ elements $\{a_1,\ldots,a_n \}\subset A$ with
	$\|a_i\|\leq \delta$ $(i=1,\ldots,n$), and every polynomial $p$ in $n$ variables without constant terms, we have
		$$\|p(a_1,\ldots,a_n)\| \leq L\|p\|_\infty.$$
	\end{defn}

Every commutative injective algebras are $Q$-algebras (\cite{V}). Actually, a commutative Banach algebra is an injective algebra if and only if it is isomorphic to a quotient of a uniform algebra by a complemented ideal (\cite{V2}).
A more qualitative result can be found in \cite{BLe} using a modern language of operator spaces.
	\begin{thm}\label{thm-inj-VN}({\bf Blecher/Le Merdy}, \cite[Theorem 5.4.5, Corollary 5.4.11]{BLe})
	Let $\A$ be a commutative injective algebra with the multiplication map $m$.
	Then $\A$ satisfies the multi-variable $(\delta, L)$-von Neuman inequality with
		$$\delta = \frac{1}{(1+\norm{m}_\eps) e}\;\; \text{and}\;\; L=1.$$
	\end{thm}

\subsection{Weighted group algebras}\label{S:weighted}

Let $G$ be a group, and let $\om : G \to (0,\infty)$ be a weight on $G$, i.e.
	$$\om(xy)\leq \om(x)\om(y),\;\; (x,y\in G).$$
The {\it weighted group algebra} $\ell^1(G,\om)$ is the convolution algebra of functions
$f$ on $G$ such that $\|f\|_{\ell^1(G,\om)}=\sum_{x\in G} |f(x)|\om(x) < \infty$.
Using the natural duality $\ell^1(G)^*=\ell^\infty(G)$, we can show that $\ell^1(G,\om)^*=\ell^\infty(G,\om^{-1})$, where
$$\ell^\infty(G,\om^{-1})=\{ \varphi \mid \varphi \om^{-1} \in \ell^\infty(G) \}$$
with
$$ \|\varphi\|_{\ell^\infty(G,\om^{-1})}=\|\varphi \om^{-1}\|_\infty.$$

In Section \ref{S:algebras}, we showed that every injective Banach algebra is isomorphic to an operator algebra. As we will show in Theorem \ref{thm-op-alg} below, the converse of the preceding statement is also true in the case of weighted group algebras. However, this requires some operator space knowledge including the Haagerup tensor product $\otimes_h$ of operator spaces. We refer the reader to \cite{BLe} or \cite{Pis1} for references. We first recall the following form of the celebrated Grothendieck's theorem.

	\begin{thm}\label{GT}
	If we equip $\ell^1(G)$ with its MAX operator space structure, then the formal identity $id: \ell^1(G)\otimes_\eps \ell^1(G) \to \ell^1(G)\otimes_h \ell^1(G)$ has norm $\le K_G$, where $K_G$ is the Grothendieck's constant.
	\end{thm}

\begin{proof}
See \cite[(1.47), (A.7)]{BLe} and \cite[(3.11)]{Pis3}.
\end{proof}

\begin{thm}\label{thm-op-alg}
Let $G$ be a group, and let $\om$ be a weight on $G$. Then $\ell^1(G,\om)$ is an injective Banach algebra
if and only if it is isomorphic to an operator algebra.
\end{thm}

\begin{proof}
The necessary part has been proven in Corollary \ref{C:injective-op alg}. For the sufficient
part, suppose that there is an operator algebra $B\subseteq B(H)$ and a bounded algebra isomorphism
 $\psi: \ell^1(G,\om) \to B$. This, in particular, implies that $\psi: \ell^1(G,\om) \to B$ is completely bounded
 when $\ell^1(G,\om)$ is given its MAX operator space structure. Now since from \cite[Theorem 2.3.2]{BLe}, the multiplication map $m: B\otimes_h B \to B$ is completely contractive, we have the following bounded map:
$$\psi^{-1}\circ m \circ (\psi \otimes \psi) : \ell^1(G,\om) {\otimes}_h \ell^1(G,\om) \to \ell^1(G,\om)$$
But it is easy to see that $\psi^{-1}\circ m \circ (\psi \otimes \psi)$ is exactly the multiplication map
	$$m: \ell^1(G,\om) {\otimes}_h \ell^1(G,\om) \to \ell^1(G,\om).$$
On the other hand, since the mapping 
$$\ell^1(G,\om) \to \ell^1(G) \ , \ f \mapsto f\om,$$
is a complete isometric surjection (here we have again given MAX operator space structure to both
$\ell^1(G,\om)$ and $\ell^1(G)$), it follows from the Grothendieck's theorem (Theorem \ref{GT}) that the formal identity $id: \ell^1(G,\om) {\otimes}_\eps \ell^1(G,\om) \to \ell^1(G,\om) {\otimes}_h \ell^1(G,\om)$ has norm $\le K_G$. This implies that the multiplication map $\ell^1(G,\om)\otimes_\eps \ell^1(G,\om) \to \ell^1(G,\om)$ is bounded, and so,  $\ell^1(G,\om)$ is injective.

\end{proof}

\subsection{Littlewood multiplier}\label{S:Littlewood mult}

Let $G$ be a (discrete) group. We let the space of {\it Littlewood multipliers}, denoted by $T^2(G)$,
to be all the functions $f : G\times G \to \C$ for which there are functions $f_1 , f_2 : G \times G \to \C$ such that
	$$f(s,t)=f_1(s,t)+f_2(s,t) \ \ \ \ (s,t \in G),$$
and
	$$ \sup_{t\in G} \sum_{s\in G} |f_1(s,t)|^2 <\infty \ \  , \ \   \sup_{s\in G} \sum_{t\in G}  |f_2(s,t)|^2<\infty.$$
We equip this space with the norm
	$$\|f\|_{T_2(G)}=\inf \left\{ \sup_{t\in G} (\sum_{s\in G} |f_1(s,t)|^2)^{1/2} + \sup_{s\in G} (\sum_{t\in G}
	|f_2(s,t)|^2)^{1/2} \right\},$$
where infimum is taken over all possible decompositions. Note that the term ``Littlewood functions" have been used for $T^2(G)$ in the literature, but we would like to use the term ``Littlewood multipliers" instead since it explains the meaning of $T^2(G)$ better.

It follows easily that $T^2(G)$, with the action of pointwise multiplication, is a symmetric Banach $\ell^\infty(G\times G)$-module. Indeed, we have the following contraction.
	\begin{equation}\label{eq-T2-module}
	\ell^\infty(G\times G) \otimes_\gamma T^2(G) \to T^2(G),\; f\otimes g \mapsto fg,
	\end{equation}
where $\otimes_\gamma$ is the projective tensor product of Banach spaces. Moreover, we have the following bounded embedding which is well-known to experts but we have presented its proof for the sake of completeness.


	\begin{prop}\label{prop-T2-embed}
	Let $G$ be a discrete group, and let $I : T_2(G) \to (\ell^1(G) \otimes_\eps \ell^1(G))^*$ be the formal identity.
	Then we have
		$$\norm{I}\le K_G.$$
	\end{prop}
\begin{proof}
For simplicity, we write $\ell^1$ instead of $\ell^1(G)$, $\ell^2$ instead of $\ell^2(G)$ and $\ell^\infty$ instead of $\ell^\infty(G)$.
We first note that since $\ell^2$ is reflexive, we have the following isometric isomorphisms
\begin{align}\label{Eq1:prop-T2-embed}
B(\ell^1, \ell^2)\cong (\ell^1\otimes^\gamma \ell^2)^*\cong (\ell^1(\ell^2))^*  \cong \ell^\infty(\ell^2),
\end{align}
where $\ell^1(\ell^2)$ and $\ell^\infty(\ell^2)$ are Banach spaces of $\ell^2$-valued 1-summable functions and bounded functions, respectively. Now
let $f_1 : G \times G \to \C$ be a function with
$$\alpha:=\sup_{t\in G} (\sum_{s\in G} |f_1(s,t)|^2)^{1/2} <\infty.$$
Then, by (\ref{Eq1:prop-T2-embed}), the associated linear map $u : \ell^1 \to \ell^2,\; g \mapsto u(g)$ given by
	$$u(g)(t) = \sum_{s\in G}g(s)f_1(s,t)$$
has the norm $\norm{u} = \alpha$ and $I(f_1)$ corresponds to $id_{2,\infty}\circ u$, where $id_{2,\infty}:\ell^2 \to \ell^\infty$ is the formal identity. Now we recall that
	$$(\ell^1(G)\otimes_\eps\ell^1(G))^* \cong \Pi_1(\ell^1(G), \ell^\infty(G))$$
and
	$$(\ell^1(G)\otimes_h \ell^1(G))^*\cong \Gamma_2(\ell^1(G), \ell^\infty(G)),$$
the space of 2-factorable operators, as Banach spaces (\cite[Proposition 5.16]{Pis1}, \cite[Chapter 7]{DJT}). Then, by Grothendieck's theorem (Theorem \ref{GT}), we have
	$$\|I(f_1)\|=\pi_1(id_{2,\infty}\circ u) \le K_G \cdot \gamma_2(id_{2,\infty} \circ u) \le K_G \cdot \alpha,$$
where $\gamma_2(\cdot)$ is the 2-factorable norm.
Similarly, for $f_2: G \times G \to \C$ with
	$$\beta:=\sup_{s\in G} (\sum_{t\in G} |f_1(s,t)|^2)^{1/2} <\infty$$
we get
	$\|I(f_2)\|\le K_G \cdot \beta,$
which gives the desired result.
\end{proof}

\subsection{Groups with polynomial growth}\label{S:Groups poly. growth}

Let $G$ be a finitely generated group with a fixed finite symmetric generating set $F$ with the identity of the group $G$.
$G$ is said to have {\it polynomial growth} if there exists a polynomial $f$ such that
	$$|F^n|\leq f(n) \ \ \ (n\in \N).$$
Here $|S|$ is the cardinality of any $S\subseteq G$ and
	$$F^n=\{u_1\cdots u_n : u_i\in F, i=1,\ldots, n \},$$
where $\{u_i\}^n_{i=1}$ is the set of generators. The least degree of any polynomial satisfying the above relation is called {\bf the order of growth} of $G$ and it is denoted by $d(G)$. It can be shown that the order of growth of $G$ does not depend on the symmetric generating set $F$, i.e. it is a universal constant for $G$.

It is immediate that finite groups are of polynomial growth. More generally, every $G$ with the property that the conjugacy class of every element in $G$ is finite has polynomial growth \cite[Theorem 12.5.17]{Pal}. Also every nilpotent group (hence an abelian group) has polynomial growth \cite[Theorem 12.5.17]{Pal}.
A deep result of M. Gromov \cite{Grom} states that every finitely generated group with polynomial
growth is virtually nilpotent i.e.\ it has a nilpotent subgroup of finite index. Moreover, there is a polynomial $f$ and a constant $0<\lambda \leq 1$ such that
\begin{align}\label{Eq:stric poly growth}
\lambda f(n) \leq |F^n| \leq f(n) \ \ \text{for all}\ \  n\in \N,
\end{align}
where $\text{deg}\, f=d(G)$. If we further assume that $G$ is nilpotent, then by
the {\bf Bass-Guivarch formula} (\cite{B}, \cite{Gui}), we can actually compute the order of growth
of $G$. More precisely, let $G$ be a finitely generated nilpotent group with lower central series
$$G = G_1 \supseteq G_2 \supseteq \ldots \supseteq G_m=\{e\}.$$
In particular, the quotient group $G_k/G_{k+1}$ is a finitely generated abelian group. Then
the order of growth of G is
\begin{align}\label{Eq:Bass-Guivarch formula}
d(G) = \sum_{k =1}^{m-1} k \ \operatorname{rank}(G_k/G_{k+1}),
\end{align}
where rank denotes the rank of an abelian group, i.e.\ the largest number of independent and torsion-free elements of the abelian group.

Using the generating set $F$ of $G$ we can define a {\it length function} $\tau_F : G \to [0, \infty)$ by
\begin{align}\label{Eq:length function}
\tau_F(x)=\inf \{n\in \N : x\in F^n \} \ \ \text{for} \ \ x \neq e, \ \ \tau_F(e)=0.
\end{align}
When there is no fear of ambiguity, we write $\tau$ instead of $\tau_F$.
It is straightforward to verify that $\tau$ is a subadditive function on $G$, i.e.
\begin{align}\label{Eq:lenght func-trai equality}
\tau(xy)\leq \tau(x)+\tau(y) \ \ \ \ (x,y\in G).
\end{align}
Note that since $F$ is symmetric, for every $x\in G$, $\tau(x)=\tau(x^{-1})$. If we combine this fact
with (\ref{Eq:lenght func-trai equality}), then a straightforward calculation
shows that
\begin{align}\label{Eq:lenght func-trai equality-double side}
|\tau(x)-\tau(y)|\leq \tau(xy)\leq \tau(x)+\tau(y) \ \ \ \ (x,y\in G).
\end{align}
We can use $\tau$ to define various weights on $G$.
More precisely, for every $0\leq \alpha \leq 1$, $\beta \geq 0$, and $C>0$,
we can define the {\it polynomial weight} $\om_\beta$ on $G$ of order $\beta$ by
\begin{align}\label{Eq:poly weight-defn}
\om_\beta(x)=(1+\tau(x))^\beta  \ \ \ \ (x\in G),
\end{align}
and the {\it exponential weight} $\sg_{\alpha, C}$ on $G$ of order $(\alpha, C)$ by
\begin{align}\label{Eq:Expo weight-defn}
\sg_{\alpha,C}(x)=e^{C\tau(x)^\alpha} \ \ \ \ (x\in G).
\end{align}

\section{Weighted group algebras isomorphic to operator algebras}

In this section, we will use $G$ to denote a finitely generated infinite group with polynomial growth.
$F$ is a fixed symmetric generating set of $G$ with the identity and $f$ and $\lambda$ refer to the polynomial and the constant satisfying \eqref{Eq:stric poly growth}.

\subsection{The case of polynomial weights}\label{S:Poly. weight-poly growth}

For some weight $\om : G \to (\delta,\infty)$ with $\delta >0$, we would like to check whether $\ell^1(G,\om)$ is an injective algebra. In order to do that we recall the co-multiplication
	$$\Gamma : \ell^\infty(G) \to \ell^\infty(G\times G),\; f\mapsto \Gamma f$$
with $\Gamma f (s,t) = f(st)$, $s,t \in G$.
Let $\Gamma_{\om} : \ell^\infty(G,\om^{-1}) \to \ell^\infty(G\times G,\om^{-1} \times \om^{-1})$ be the extension of $\Gamma$ to $\ell^\infty(G,\om^{-1})$.
Consider the isometries
	$$P: \ell^\infty(G) \to \ell^\infty(G,\om^{-1}), \; f \mapsto f\om$$
and
	$$R: \ell^\infty(G\times G,\om^{-1} \times \om^{-1}) \to \ell^\infty(G\times G),\; F\mapsto F\cdot (\om^{-1} \times \om^{-1}).$$

We define the operator $\tilde{\Gamma} : \ell^\infty(G) \to \ell^\infty(G\times G)$ so that the following diagram commutes:
\[
\xymatrix{
\ell^\infty(G,\om^{-1}) \ar@<.5ex>[rr]^{\Gamma_{\om}}
& & \ell^\infty(G\times G,\om^{-1} \times \om^{-1})  \ar[d]^{R}   \\
\ell^\infty(G) \ar@<.5ex>[u]^{P}  \ar@<.5ex>[rr]^{\tilde{\Gamma}}
& & \ell^\infty(G\times G)  }
\]
Hence
	\begin{equation}\label{eq-modified-comulti}
	\tilde{\Gamma}(f)=\Om \Gamma(f) \ \ (f\in \ell^\infty(G)),
	\end{equation}
where
\begin{align}\label{Eq:weight-littlewood mult}
\Om:=\frac{\Gamma (\om)}{\om \times \om}.
\end{align}
Now $\ell^1(G,\om)$ is an injective algebra if and only if
the multiplication map
	$$m : \ell^1(G,\om)\otimes_\eps \ell^1(G,\om) \to \ell^1(G,\om)$$
is bounded, or equivalently, $\tilde{\Gamma}$ extends to a bounded map
	$$\tilde{\Gamma} : \ell^\infty(G) \to (\ell^1(G) \otimes_\eps \ell^1(G))^*.$$
Note that we have
	$$\norm{m}_\eps = \norm{\tilde{\Gamma}}.$$
An application of Littlewood multiplier argument gives the following positive results on
the weighted group algebra $\ell^1(G,\om_\beta)$,
where $\om_\beta$ is the polynomial weight defined in (\ref{Eq:poly weight-defn}).

\begin{thm}\label{T:ploy weight-operator alg}
$\ell^1(G,\om_\beta)$ is an injective algebra
if one of the following conditions holds:\\
$(i)$ $\lambda=1$ and $\beta> \frac{d(G)}{2}$;\\
$(ii)$ $0< \lambda < 1$ and $\beta> \frac{d(G)+1}{2}$.\\
Moreover, we have
\begin{align}\label{Eq:poly weight-Hag norm}
\norm{m}_\eps \le K_G\min \{1,2^{\beta-1}\}\left[1+\sum_{n=1}^\infty \frac{f(n)-\lambda f(n-1)}{(1+n)^{2\beta}} \right]^{1/2}.
\end{align}

\end{thm}

\begin{proof}
Let $\Om_\beta:=\frac{\Gamma (\om_\beta)}{\om_\beta \times \om_\beta}$.
We will first show that $\Om_\beta \in T^2(G)$.
For every $x,y\in G$, we have
\begin{eqnarray*}
\Om_\beta(x,y) &=& \frac{\om_\beta(xy)}{\om_\beta(x)\om_\beta(y)} \\
&=& \frac{(1+\tau(xy))^\beta}{(1+\tau(x))^\beta(1+\tau(y))^\beta} \\
&\leq & \frac{(1+\tau(x)+\tau(y))^\beta}{(1+\tau(x))^\beta(1+\tau(y))^\beta} \\
&\leq & \frac{A_\beta[(1+\tau(x))^\beta+(1+\tau(y))^\beta]}{(1+\tau(x))^\beta(1+\tau(y))^\beta} \ \ \ (*)\\
&=& \frac{A_\beta}{(1+\tau(x))^\beta}+ \frac{A_\beta}{(1+\tau(y))^\beta},
\end{eqnarray*}
where $A_\beta =\min \{1,2^{\beta-1}\}$ and the inequality ($*$) follows
from the classical inequality
$$(a+b)^\beta \leq A_\beta (a^\beta+b^\beta) \ \ \ (a,b\geq 0).$$
Hence there is the function $u\in \ell^\infty(G\times G)$ with $\|u\|_\infty \leq 1$ such that
$$\Om_\beta(x,y)=u(x,y)\left[\frac{A_\beta}{(1+\tau(x))^\beta}+ \frac{A_\beta}{(1+\tau(y))^\beta}\right] \ \ \ (x,y\in G).$$
Thus, by the definition of $T^2(G)$ and \eqref{eq-T2-module},
\begin{align}\label{Eq:poly weight-littlewood bound}
\|\Om_\beta\|_{T^2(G)} \leq A_\beta\left(\sum_{x\in G} \frac{1}{(1+\tau(x))^{2\beta}}\right)^{1/2}.
\end{align}
Hence it suffices to see when $\sum_{x\in G} \displaystyle \frac{1}{(1+\tau(x))^{2\beta}}$ is finite.
To see this, from our hypothesis and (\ref{Eq:stric poly growth}), we have
\begin{eqnarray*}
\sum_{x\in G} \frac{1}{(1+\tau(x))^{2\beta}} &=&
\sum_{n=0}^\infty \sum_{\tau(x)=n} \frac{1}{(1+n)^{2\beta}} \\
&= &  1+\sum_{n=1}^\infty \sum_{x\in F^n \setminus F^{n-1}} \displaystyle \frac{1}{(1+n)^{2\beta}} \\
&\leq & 1+\sum_{n=1}^\infty \frac{f(n)-\lambda f(n-1)}{(1+n)^{2\beta}},
\end{eqnarray*}
where the series in the last line converges if $\lambda =1$ and $2\beta> d$ or
$0< \lambda < 1$ and $2\beta> d+1$. Moreover, in either case, we have
$$\sum_{x\in G} \frac{1}{(1+\tau(x))^{2\beta}} \leq 1+\sum_{n=1}^\infty \frac{f(n)-\lambda f(n-1)}{(1+n)^{2\beta}}.$$
Hence, by Proposition \ref{prop-T2-embed} and \eqref{eq-T2-module},
\begin{align*}
\norm{\tilde{\Gamma}(f)}_{ (\ell^1(G) \otimes_\eps \ell^1(G))^*}
&\le K_G\norm{\tilde{\Gamma}(f)}_{T^2(G)}\\
&\le K_G\norm{\Om_\beta}_{T^2(G)} \norm{\Gamma(f)}_\infty\\
&\le K_GA_\beta \left[1+\sum_{n=1}^\infty \frac{f(n)-\lambda f(n-1)}{(1+n)^{2\beta}}\right]^{1/2}\norm{f}_\infty
\end{align*}
for any $f\in \ell^\infty(G)$.

\end{proof}

\subsection{The case of exponential weights}\label{S:Exp. weight-poly growth}

In this section we will study when the weighted group algebra $\ell^1(G,\sg_{\alpha, C})$ is an injective algebra,
where $\sg_{\alpha, C}$ is the exponential weight defined in (\ref{Eq:Expo weight-defn}).
If we consider the same additional function $\Om =\frac{\Gamma (\om)}{\om \times \om}$,
then it is not clear this time whether we can split the function into two parts with a suitable square summability.
However, we can majorize the function with a similar one coming from a polynomial weight.
Let us begin with a technical lemma.

\begin{lem}\label{L:estimate poly-expo weight}
Let $0<\alpha<1$, $C>0$ and take $\beta \geq \displaystyle \max\left\{1, \frac{6}{C\alpha(1-\alpha)} \right\}$.
Define the functions $p: [0,\infty) \to \Real$ and $q : (0,\infty)\to \Real$ by
\begin{align}\label{Eq:estimate poly-expo weight}
p(x)=Cx^\alpha-\beta \ln (1+x) \ \ , \ \ q(x)=\frac{p(x)}{x}.
\end{align}
Then $p$ is increasing and $q$ is decreasing on $\displaystyle  \left[\left(\frac{\beta^2}{C\alpha(1-\alpha)}\right)^{1/\alpha} , \infty \right)$.
\end{lem}

\begin{proof}
We have
$$p'(x)=C\alpha x^{\alpha-1}-\frac{\beta}{1+x}=\frac{C\alpha x^{\alpha-1}+C\alpha x^\alpha-\beta}{1+x}.$$
Hence $p'(x)\geq 0$ if $C\alpha x^{\alpha}-\beta \geq 0$. This implies that
\begin{align}\label{Eq:1}
p\ \text{is increasing on}\  \left[\left(\frac{\beta}{C\alpha}\right)^{1/\alpha} , \infty \right).
\end{align}
Now consider $q(x)=\displaystyle Cx^{\alpha-1}-\frac{\beta \ln (1+x)}{x}$. Then
\begin{align}\label{Eq:2}
q'(x)= \frac{C(\alpha-1) x^{\alpha}-\displaystyle\frac{\beta x}{1+x}+\beta \ln (1+x)}{x^2}=
\frac{h(x)-\displaystyle\frac{\beta x}{1+x}}{x^2},
\end{align}
where
$$h(x):=C(\alpha-1) x^{\alpha}+\beta \ln (1+x).$$
Hence, in order to find an interval for which $q'(x)\leq 0$, it suffices to see when
$h(x)\leq 0$. We have
$$ h'(x)= \frac{C\alpha (\alpha-1) x^{\alpha}+C\alpha (\alpha-1) x^{\alpha-1}+\beta}{1+x}.$$
Thus if we put
$$C_1=C\alpha (1-\alpha),$$
then $h'(x)\leq 0$ whenever $-C_1x^\alpha+\beta \leq 0$, or equivalently, $x\geq \displaystyle (\frac{\beta}{C_1})^{1/\alpha}$.
Hence
\begin{align}\label{Eq:3}
h\ \text{is decreasing on}\  \left[\left(\frac{\beta}{C_1}\right)^{1/\alpha} , \infty \right).
\end{align}
Now since, by hypothesis, $\beta \geq 1$, we have
	$$\displaystyle (\frac{\beta^2}{C_1})^{1/\alpha} \geq (\frac{\beta}{C_1})^{1/\alpha},$$
and so by (\ref{Eq:3}),
	$$h(x)\leq h((\frac{\beta^2}{C_1})^{1/\alpha}) \ \ \text{whenever}\ \ x\geq  (\frac{\beta^2}{C_1})^{1/\alpha}.$$
This implies that if $x\geq \displaystyle (\frac{\beta^2}{C_1})^{1/\alpha}$, then
\begin{eqnarray*}
h(x) &\leq & C(\alpha-1) \frac{\beta^2}{C_1}+\beta \ln (1+(\frac{\beta^2}{C_1})^{1/\alpha}) \\
&=&   \beta[\ln (1+(\frac{\beta^2}{C_1})^{1/\alpha})-\frac{\beta}{\alpha}].
\end{eqnarray*}
On the other hand, since $\beta \geq \displaystyle \frac{6}{C\alpha(1-\alpha)}=\frac{6}{C_1}$, we have
$\displaystyle \frac{\beta^2}{C_1} \leq \frac{\beta^3}{6}$. Hence, considering the fact that $1/\alpha >1$,
\begin{eqnarray*}
1+ (\frac{\beta^2}{C_1})^{1/\alpha}
& \leq & 1+(\frac{\beta^3}{3!})^{1/\alpha} \\
& \leq & \left(1+\frac{\beta^3}{3!}\right)^{1/\alpha} \\
&\leq & \left(\sum_{n=0}^\infty \frac{\beta^n}{n!}\right)^{1/\alpha}\\
&=& \displaystyle e^{\frac{\beta}{\alpha}}.
\end{eqnarray*}
Therefore
$$\ln (1+(\frac{\beta^2}{C_1})^{1/\alpha})-\frac{\beta}{\alpha} \leq 0.$$
Hence $ h(x)\leq 0$ if $x\geq  (\frac{\beta^2}{C_1})^{1/\alpha}$. By (\ref{Eq:2})
\begin{align}\label{Eq:4}
 q(x)\ \ \text{is decreasing on}\ \ \left[\left(\frac{\beta^2}{C_1}\right)^{1/\alpha}, \infty \right).
\end{align}
The final result follows from (\ref{Eq:1}) and the fact that $\displaystyle (\frac{\beta^2}{C_1})^{1/\alpha} \geq
(\frac{\beta}{C\alpha})^{1/\alpha}$.
\end{proof}

\begin{thm}\label{T:Expo weight-bounded-Poly weight}
Suppose that $0<\alpha<1$, $C>0$, and $\beta \geq \displaystyle \max\left \{1, \frac{6}{C\alpha(1-\alpha)} \right \}$.
Let $p$ and $q$ be the functions defined in (\ref{Eq:estimate poly-expo weight})
and consider the function $\om : G \to (0,\infty)$ defined by
$$\om(x)=e^{p(\tau(x))}=e^{\tau(x)q(\tau(x))} \ \ \ (x\in G).$$
Then
$$\om(xy)\leq M\om(x)\om(y)  \ \ \ \ (x,y\in G),$$
where
\begin{align}\label{Eq:bound-poly expo weight}
M=\max \{ e^{p(t)-p(s)-p(r)}: t,s,r \in [0,4K]\cap \Z \}.
\end{align}
and
\begin{align}\label{Eq:increasing-decreasing interval}
K=\left(\frac{\beta^2}{C\alpha(1-\alpha)}\right)^{1/\alpha}.
\end{align}
\end{thm}

\begin{proof}
By Lemma \ref{L:estimate poly-expo weight},
$p$ is increasing and $q$ is decreasing on $[K,\infty)$.
We will prove the statement of the theorem considering various
cases:\\
{\it Case I:} $\max\{ \tau(x), \tau(y)\}\leq 2K$. In this case, $\tau(xy)\leq \tau(x)+\tau(y)\leq 4K$.
Hence
$$\frac{\om(xy)}{\om(x)\om(y)}=e^{p(\tau(xy))-p(\tau(x))-p(\tau(y))} \leq M.$$
{\it Case II:} $\max\{ \tau(x), \tau(y)\}>  2K$ and $\min\{ \tau(x), \tau(y)\}\leq  K$.
Without loss of generality, we can assume that $\tau(x)>2K$ and $\tau(y)\leq K$.
Then, by (\ref{Eq:lenght func-trai equality-double side}),
$$\tau(x)+\tau(y)\geq \tau(xy)\geq \tau(x)-\tau(y)\geq 2K-K=K.$$
Thus, by Lemma \ref{L:estimate poly-expo weight},
\begin{eqnarray*}
\om(xy) & = & e^{p(\tau(xy))} \\
& \leq &  e^{p(\tau(x)+\tau(y))} \\
&=& e^{(\tau(x)+\tau(y))q(\tau(x)+\tau(y))}\\
&=& e^{\tau(x)q(\tau(x)+\tau(y))}e^{\tau(y)q(\tau(x)+\tau(y))}\\
&\leq& e^{\tau(x)q(\tau(x))}e^{Kq(K)}\\
&=& \om(x)\om(y)e^{p(K)-p(\tau(y))} \\
&\leq & M \om(x)\om(y).
\end{eqnarray*}
{\it Case III:} $\min\{ \tau(x), \tau(y)\}> K$ and $\tau(xy)\leq K$.
In this case, we have
\begin{eqnarray*}
\om(x)\om(y) & = &  e^{p(\tau(x))+p(\tau(y))} \\
& \geq & e^{2p(K)} \\
&=& e^{2p(K)-p(\tau(xy))}\om(xy) \\
&\geq & \frac{1}{M} \om(xy).
\end{eqnarray*}
Hence
$$ \om(xy)\leq M \om(x)\om(y).$$
{\it Case IV:} $\min\{ \tau(x), \tau(y), \tau(xy)\}> K$.
In this case, by Lemma \ref{L:estimate poly-expo weight}, we have
\begin{eqnarray*}
\om(xy) & = & e^{p(\tau(xy))} \\
& \leq &  e^{p(\tau(x)+\tau(y))} \\
&=& e^{(\tau(x)+\tau(y))q(\tau(x)+\tau(y))}\\
&=& e^{\tau(x)q(\tau(x)+\tau(y))}e^{\tau(y)q(\tau(x)+\tau(y))}\\
&\leq& e^{\tau(x)q(\tau(x))}e^{\tau(y)q(\tau(y)}\\
&\leq& \om(x)\om(y).
\end{eqnarray*}
Therefore by comparing the above four cases and considering the fact that
$M\geq e^{-p(0)}=1$, it follows that for
every $x,y\in G$,
$$\om(xy)\leq M\om(x)\om(y).$$
\end{proof}

We are now ready to show when the weighted group algebras of exponential weights are injective algebras.

\begin{thm}\label{T:Expo weight-2-summing-alg}
Suppose that $0< \alpha <1$ and $C>0$. Then $\ell^1(G,\sg_{\alpha, C})$ is a 2-summing algebra.
Moreover, we have
\begin{align}\label{Eq:Expo weight-g_2-norm}
\norm{m}_\eps \le K_G M2^{\beta-1}\left [1+\sum_{n=1}^\infty \frac{f(n)-\lambda f(n-1)}{(1+n)^{2\beta}}\right]^{1/2},
\end{align}
where
	$$\beta = \displaystyle \max\left \{1, \frac{6}{C\alpha(1-\alpha)}, \frac{d+(1-\delta_1(\lambda))}{2} \right\}$$
($\delta_1$ is the Dirac function at 1)
and $M$ is the constant (depending on $\alpha$, $\beta$ and $C$) defined in (\ref{Eq:bound-poly expo weight}).
\end{thm}

\begin{proof}
We define a function $\om : G \to (0,\infty)$ by
	$$\om(x)=\frac{\sg_{\alpha, C}(x)}{\om_\beta(x)}=e^{C\tau(x)^\alpha-\beta \ln (1+\tau(x))} \ \ (x\in G),$$
where $\om_\beta$ is the polynomial weight defined in (\ref{Eq:poly weight-defn}).
Then by Theorem \ref{T:Expo weight-bounded-Poly weight},
	$$\om(xy)\leq M\om(x)\om(y)  \ \ \ \ (x,y\in G),$$
where $M$ is the constant defined in (\ref{Eq:bound-poly expo weight}). Therefore if we let
$$\displaystyle \Sigma_{\alpha, C} :=\frac{\Gamma (\sg_{\alpha, C})}{\sg_{\alpha, C}\times \sg_{\alpha, C}}  \ \ \text{and} \ \ \Om_\beta:=\frac{\Gamma (\om_\beta)}{\om_\beta \times \om_\beta},$$ then
$$\Sigma_{\alpha, C} \leq M\Om_\beta \leq  M\left[\frac{2^{\beta-1}}{(1+\tau(x))^\beta}+ \frac{2^{\beta-1}}{(1+\tau(y))^\beta}\right].$$
A similar argument to the one presented in the proof of Theorem \ref{T:ploy weight-operator alg}
shows that
$$\norm{m}_\eps \le K_G \|\Sigma_{\alpha, C} \|_{T^2(G)} \le K_G M2^{\beta-1}\left[1+\sum_{n=1}^\infty \frac{f(n)-\lambda f(n-1)}{(1+n)^{2\beta}}\right]^{1/2}.$$
In particular, $\ell^1(G,\sg_{\alpha, C})$ is an injective algebra.
\end{proof}

We can actually show exactly when the weighted group algebras of exponential weight are isomorphic to an operator algebra.

\begin{thm}\label{T:Expo weight-operator alg}
Suppose that $0\leq \alpha \leq 1$ and $C>0$.
Then $\ell^1(G,\sg_{\alpha, C})$ is isomorphic to an operator algebra if and only if $0<\alpha<1$.
\end{thm}

\begin{proof}
The case $0<\alpha<1$ is done already.

If $\alpha=0$, then $\ell^1(G,\sg_{\alpha, C})\cong \ell^1(G)$ which is known to be non-Arens
regular (\cite[Theorem 8.11]{DL}) and so, it is not an operator algebra.
Now suppose that $\alpha=1$. For every $m,n\geq 2$, take $a_{m,n}\in F^{m+n}\setminus F^{m+n-1}$
(this is possible because $G$ is infinite).
Hence there are $x_n\in F^n$ and $y_m\in F^m$ such that
$$a_{m,n}=x_ny_m.$$
Moreover, since $a_{m,n}\in F^{m+n}\setminus F^{m+n-1}$, we have
$$x_n\in F^n \setminus F^{n-1}\ \text{and}\ \ y_m\in F^m\setminus F^{m-1}.$$
Therefore
$$\tau(a_{m,n})=m+n \ , \ \tau(x_n)=n \ , \ \tau(y_m)=m.$$
Hence
$$\frac{\sg_{1,C}(x_ny_m)}{\sg_{1,C}(x_n)\sg_{1,C}(y_m)}=\frac{e^{C\tau(x_ny_m)}}{e^{C(\tau(x_n)+\tau(y_m))}}
=\frac{e^{Cm+Cn}}{e^{C(n+m)}}=1.$$
Thus
$$\lim_{n\to \infty}\lim_{m\to \infty} \frac{\sg_{1,C}(x_ny_m)}{\sg_{1,C}(x_n)\sg_{1,C}(y_m)}=1,$$
which implies from \cite[Theorem 8.11]{DL} that $\ell^1(G,\sg_{1,C})$ is not Arens regular, and so, it
is not an operator algebra.

\end{proof}

\begin{rem}
We would like to point out that the upper-bounded estimate obtained in (\ref{Eq:Expo weight-g_2-norm})
goes to $\infty$ as $\alpha$ approaches either $0$ or $1$ (this happens because $\beta \to \infty$).
This coincides with the result obtained in the statement of Theorem \ref{T:Expo weight-operator alg}
since as $\alpha \to 0$ ($\alpha \to 1$, respectively), the weight $\sg_{\alpha,C} \to \sg_{0,C}=e^C$
($\sg_{\alpha,C} \to \sg_{1,C}$, respectively) and we showed there that neither $\ell^1(G,e^C)$ nor
$\ell^1(G,\sg_{1,C})$ is isomorphic to an operator algebra, and so, $\norm{m}_\eps$ is not bounded.
\end{rem}

\section{Remarks on $Q$-algebras and operator space versions}

The weighted group algebras in sections \ref{S:Poly. weight-poly growth} and \ref{S:Exp. weight-poly growth} are injective algebras,
but not isomorphic to $Q$-algebras since they are non-commutative in general.
However, their algebraic centers are actually isomorphic to $Q$-algebras.
Indeed, the injectivity of the tensor product tells us that the algebraic center is also an injective algebra with the smaller norm of the multiplication map. Then, the result in \cite{V} implies that they are isomorphic to $Q$-algebras.
Moreover, Theorem \ref{thm-inj-VN} allows us to determine $(\delta,L)$ for the corresponding multi-variable von Neumann inequality. Thus we have the following. We note that for an algebra $A$, we denote $ZA$ to be its algebraic center.

\begin{cor}\label{T:Poly weight-center-Q alg}
$Z\ell^1(G,\om_\beta)$ is isomorphic to a $Q$-algebra if one of the following conditions holds:\\
$(i)$ $\lambda=1$ and $\beta> \frac{d(G)}{2}$;\\
$(ii)$ $0< \lambda < 1$ and $\beta> \frac{d(G)+1}{2}$.

In this case $Z\ell^1(G,\om_\beta)$ satisfy multi-variable $(\delta,L)$-von Neumann inequality with
	$$\delta = e^{-1} \left(1 + K_G\min \{1,2^{\beta-1}\} \left[1+\sum_{n=1}^\infty \frac{f(n)-\lambda f(n-1)}{(1+n)^{2\beta}} \right]^{1/2}\right)^{-1}\;\;\text{and}\;\; L=1.$$
\end{cor}

We have a corresponding result for exponential weights.

\begin{cor}\label{T:Expo weight-center-Q alg}
Suppose that $0< \alpha <1$ and $C>0$. Then $Z\ell^1(G,\sg_{\alpha, C})$ is isomorphic to a $Q$-algebra.
In this case, $Z\ell^1(G,\sg_{\alpha, C})$ satisfy multi-variable $(\delta,L)$-von Neumann inequality with
	$$\delta = e^{-1}\left(1+K_GM2^{\beta-1}\left [1+\sum_{n=1}^\infty \frac{f(n)-\lambda f(n-1)}{(1+n)^{2\beta}}\right]^{1/2}\right)^{-1}\;\;\text{and}\;\; L=1,$$
where
	$$\beta = \displaystyle \max\left \{1, \frac{6}{C\alpha(1-\alpha)}, \frac{d+(1-\delta_1(\lambda))}{2} \right\}$$
and $M$ is the constant defined in (\ref{Eq:bound-poly expo weight}).
\end{cor}

We end this section with a remark on operator space versions.
Most of the results in this paper have their operator space versions available following the approach in \cite{CGLSS}.
For example, the estimates on $\norm{\Om_\beta}_{T^2(G)}$ in Theorem \ref{T:ploy weight-operator alg} tells us that
$\ell^1(G,\om_\beta)$ with the maximal operator space structure is completely isomorphic to an operator algebra.
But in the case of operator spaces we need to show that the algebra multiplication map $m$ extends to
a completely bounded maps on the Haagerup tensor product,
so that Littlewood multiplier theory has to be developed upto the level of operator spaces as in \cite{CGLSS}.

\section{examples}\label{S:Examples}

\subsection{The d-dimensional integers $\Z^d$}
A usual choice of generating set is
	$$F=\{(x_1,\ldots, x_d) \mid x_i\in \{-1,0,1\} \},$$
It is straightforward to see that
	$$\tau((x_1,\ldots, x_d)) = \max\{\abs{x_1}, \ldots, \abs{x_d}\}$$
and for every $n\in \N$,
	$$F^n=\{(x_1,\ldots, x_d) \mid x_i\in \{-n,\ldots, 0, \ldots, n\} \}.$$
Thus we get
	$$|F^n|=(2n+1)^d \ \ \ (n=0,1,2,\ldots)$$
and the order of growth of $\Z^d$ is $d$ with $f(n) = (2n+1)^d$ and $\lambda = 1$.
It follows from Theorem \ref{T:ploy weight-operator alg} that
$\ell^1(\Z^d,\om_\beta)$ is isomorphic to an operator algebras if $\beta > \frac{d}{2}$.
Moreover, we have
	\begin{align*}
		\sum_{n=1}^\infty \frac{f(n)-\lambda f(n-1)}{(1+n)^{2\beta}}
		& = \sum_{n=1}^\infty \frac{(2n+1)^d-(2n-1)^d}{(1+n)^{2\beta}}\\
		& \le \sum_{n=1}^\infty \frac{2d(2n+2)^{d-1}}{(1+n)^{2\beta}}
		= d2^d \sum_{n=1}^\infty (1+n)^{d-1-2\beta}\\
		& \le d2^d \int_1^\infty x^{d-1-2\beta}\,dx
		= \frac{d2^d}{2\beta-d}.
	\end{align*}
Since $\Z^d$ is an abelian group, Theorem \ref{T:Poly weight-center-Q alg} tells us that
$\ell^1(\Z^d,\om_\beta)$ is actually a $Q$-algebra and it satisfies multi-variable von Neumann inequality for $L=1$ and
	$$\delta=e^{-1}\left\{1+K_G \min\{1,2^{\beta-1}\} \left[1+\frac{d2^d}{2\beta-d} \right]^{1/2} \right\}^{-1}.$$
On the other hand, $\ell^1(\Z^d,\om_\beta)$ fails to be an injective algebra if $\beta\le \frac{d}{2}$ (\cite{CGLSS}).

Now let $\sg_{\alpha, C}$ be the exponential weight on $\Z^d$ defined in (\ref{Eq:Expo weight-defn}).
Theorem \ref{T:Expo weight-center-Q alg} tells us that $\ell^1(\Z^d,\sg_{\alpha, C})$ is a $Q$-algebra and
it satisfies multi-variable von Neumann inequality for $L=1$ and
	$$\delta = e^{-1}\left(1+K_GM2^{\beta-1} \left [ 1+\frac{d2^d}{2\beta-d} \right]^{1/2}\right)^{-1},$$
where
	$$\beta = \displaystyle \max\left \{1, \frac{6}{C\alpha(1-\alpha)}, \frac{d}{2} \right\}$$
and $M$ is the constant defined in (\ref{Eq:bound-poly expo weight}).

A case of particular interest happens when we let
	$$d=1 \ \ \ \text{and}\ \ \ C=\frac{6}{\alpha(1-\alpha)}.$$
In this case, we can choose $\beta=1$. Also if $K$ is the constant defined in (\ref{Eq:increasing-decreasing interval}), then it is easy to see that $0<K<1/6$. Hence $M=1$, and so we get
	$$\delta = \frac{1}{e(1+\sqrt{3}K_G)}.$$

\subsection{The 3-dimensional discrete Heisenberg group $\mathbb{H}_3(\Z)$}
We recall that the 3-dimensional discrete Heisenberg group
$\mathbb{H}_3(\Z)$ is a semidirect product of $\Z^2$ with $\Z$ and the product is defined as follows:
$$(a_1,b_1,c_1)\cdot(a_2,b_2,c_2)=(a_1+a_2,b_1+b_2,a_1b_2+c_1+c_2) \ \ \ (a_i,b_i,c_i)\in \mathbb{H}_3(\Z).$$
If we identify $\Z$ with the subgroup $\{(0,0,c) : c\in \Z \}$, then it is easy to see that
$\mathbb{H}_3(\Z) / \Z \cong \Z^2$. Hence $\mathbb{H}_3(\Z)$ is a 2-step nilpotent group and by the Bass-Guivarch formula (\ref{Eq:Bass-Guivarch formula}) we have
	$$d(\mathbb{H}_3(\Z))=4.$$
Hence if we let $\om_\beta$ be the polynomial weight on $\mathbb{H}_3(\Z)$,
then, $\ell^1(\mathbb{H}_3(\Z), \om_\beta)$ is isomorphic to an operator algebra provided that
	$$\beta > \frac{4+1}{2}=\frac{5}{2}.$$
Moreover, $Z\ell^1(\mathbb{H}_3(\Z),\om_\beta)$ satisfies multi-variable von Neumann inequality.
On the other hand, the restriction of $\om_\beta$ to $\Z$ will be a weight equivalent to the weight $\om'_\beta (c)=(1+|c|)^\beta$. Hence $\ell^1(\mathbb{H}_3(\Z),\om_\beta)$ has a closed subalgebra which is isomorphic to $\ell^1(\Z,\om'_\beta)$. Thus it follows from the result of Varopoulos \cite{V} that $\ell^1(\mathbb{H}_3(\Z),\om_\beta)$ fails to be an injective algebra if $\beta\leq 1/2$.

\subsection{The free group with two generators $\Fb_2$}

In this subsection we will show that $\ell^1(\Fb_2, \om_\beta)$ is not an injective algebra for any $\beta>0$.
Since $\Fb_2$ is one of the typical examples of exponentially growing groups,
this gives evidence to suggest that the condition of polynomial growth on the group
is necessary for a weighted group to be realizable as an operator algebras.

Recall also the Rudin-Shapiro polynomials defined in the following recursive way (\cite[Chapter 4]{Bor}).
	$$P_0(z) := 1,\;\; Q_0(z) :=1$$
and for $k\ge 0$
	$$P_{k+1}(z) := P_k(z) + z^{2^k}Q_k(z),\;\; Q_{k+1}(z) := Q_k(z) - z^{2^k}P_k(z).$$
By an induction on $k$, it is straightforward to check that the coefficients of $P_k$ are $\pm 1$, deg$P_k = \text{deg} Q=2^k-1$ and
$$ |P_k(z)|^2+|Q_k(z)|^2=2^{k+1} \ \ \ (z\in \mathbb{T}).$$ Hence
	$$\norm{P_k}_{L^\infty(\mathbb{T})} \le \sqrt{2^{k+1}}.$$
Using the following contraction (actually it is a metric surjection due to Nehari's theorem, see \cite[Section 6]{Pis2} for example)
	$$Q : L^\infty(\mathbb{T}) \to B(\ell^2), \;\; f \mapsto (\widehat{f}(-(i+j)))_{i,j\in \Z}$$
we get a sequence of Hankelian matrices
	$$A_{2^k} = Q(\overline{P}_k),\; k\ge 0,$$ where $A_{2^k}$ is a $2^k\times 2^k$ matrix with entries $\pm 1$ satisfying
	$$\norm{A_{2^k}}_{\text{op}} \le \sqrt{2^{k+1}},$$
where $\norm{\cdot}_{\text{op}}$ means the operator norm.

	\begin{thm}
	$\ell^1(\Fb_2, \om_\beta)$ is not an injective algebra for any $\beta>0$.
	\end{thm}
\begin{proof}

Let $g_1$ and $g_2$ be two generators of $\Fb_2$, and let $d$ be an even positive integer with $d> 2\beta$. Consider the following subsets of $\Fb_2$.
	$$I^d_n = \{g^{x_1}_1g^{x_2}_2g^{x_3}_1\cdots g^{x_d}_2 : 1\le x_i\le n \ \text{for}\ i=1,\ldots,d \},$$
Now we recall the function $\Om_\beta$ defined by
	$$\Om_\beta(g, g') = \frac{\om_\beta(gg')}{\om_\beta(g)\om_\beta(g')},\; \ \ (g,g'\in \Fb).$$
Let $\Om^n_\beta = \Om_\beta 1_{I^d_n \times I^d_n}$.
When $g, g'\in I^d_n$ are given by $g = g^{x_1}_1g^{x_2}_2g^{x_3}_1\cdots g^{x_d}_2$ and $g' = g^{y_1}_1g^{y_2}_2g^{y_3}_1\cdots g^{y_d}_2$ for $x_i, y_j\ge 1$, then we have
	$$\Om^n_\beta(g,g') = \left(\frac{1+x_1+\cdots+x_n+y_1+\cdots+y_n}{(1+x_1+\cdots+x_n)(1+y_1+\cdots+y_n)}\right)^\beta.$$
By a similar estimate as in \cite[Theorem 6.1]{CGLSS} we get
	$$\norm{\Om^n_\beta}_{\text{op}} \ge 2^{-\beta}n^{\frac{d}{2}}\left(\sum_{1\le x_1,\cdots, x_d\le n}\frac{1}{(1+x_1+\cdots+x_d)^{2\beta}}\right)^{\frac{1}{2}}.$$

Now using the Rudin-Shapiro polynomial, we have a sequence of matrices $A_n \in M_n$, $n=2^k$ $(k=1,2,\ldots)$ satisfying the following conditions:
	\begin{enumerate}
		\item $A_n = (a^n_{i+j})^n_{i,j=1}$ with $a^n_i \in \{\pm 1\}$.
		\item $\norm{A_n}_{\text{op}} \le \sqrt{2n}.$
	\end{enumerate}

We consider $b=(b_h)_{h\in \Fb_2}$ given by
	$$\begin{cases}b_{gg'} = a^n_{x_1+y_1}\cdots a^n_{x_d + y_d}
	& \text{for}\;\; g=g^{x_1}_1g^{x_2}_2g^{x_3}_1\cdots g^{x_d}_2,\;g'=g^{y_1}_1g^{y_2}_2g^{y_3}_1\cdots g^{y_d}_2,\; x_i, y_j\ge 1\\
	b_h = 0 & \text{elsewhere.}
	\end{cases}$$
In other words, the matrix $\left[b_{gg'}\right]_{g,g'\in I^d_n}$ is nothing but the $d$-tensor power of the matrix $\left[a^n_{x+y}\right]_{1\le x,y\le n}$. Thus it follows from \cite[Theorem 3.1 and Corollary 3.2]{Spr1}, (\ref{eq-modified-comulti}) and $\|b\|_{\text{op}}\le (2n)^{\frac{d}{2}}$ that
	\begin{align*}
	\norm{\widetilde{\Gamma}} & \ge \norm{\widetilde{\Gamma}(b)}_{(\ell^1(G) \otimes_\eps \ell^1(G))^*} = \|\Gamma(b) \Om \|_{(\ell^1(G) \otimes_\eps \ell^1(G))^*} \\
& \ge K^{-1}_G \|\Gamma(b) \Om \|_{(\ell^1(G) \otimes_h \ell^1(G))^*} = K^{-1}_G \norm{\left[b_{gg'}\Om^n_\beta(g,g')\right]_{g,g'\in I^d_n}}_{\text{Schur}} \\
&	\ge K^{-1}_G\norm{\left[b_{gg'}\right]_{g,g'\in I^d_n}}^{-1}_{\text{op}} \norm{\Om^n_\beta}_{\text{op}} \ \ \ \ \ \ (*)\\
	& \ge K^{-1}_G(2n)^{-\frac{d}{2}}2^{-\beta}n^{\frac{d}{2}}
	\left(\sum_{1\le x_1,\cdots, x_d\le n}\frac{1}{(1+x_1+\cdots+x_d)^{2\beta}}\right)^{\frac{1}{2}}\\
	& = K^{-1}_G2^{-\frac{d}{2}}2^{-\beta}\left(\sum_{1\le x_1,\cdots, x_d\le n}\frac{1}{(1+x_1+\cdots+x_d)^{2\beta}}\right)^{\frac{1}{2}} \\
  & \longrightarrow \infty \ \text{as}\ n = 2^k\to \infty \ \text{since}\ 2\beta < d.
	\end{align*}
Hence $\ell^1(\Fb_2, \om_\beta)$ is not an injective algebra for any $\beta>0$. Note that in ($*$), we are using the fact that the Schur product of $[b_{gg'}]$ with itself is the matrix with all entries 1 which is the identity in the Schur product.
\end{proof}

\end{document}